\documentclass[a4paper,12pt]{amsart}
\usepackage{amsfonts,amsthm,amssymb,amsmath,amscd}
\usepackage[colorlinks=true,linkcolor=blue,citecolor=blue]{hyperref}

\usepackage{tikz-cd}

\usepackage{tabularx}
\usepackage{tcolorbox}
\usepackage[english]{babel}
\usepackage{amssymb,amsmath}
\usepackage{xcolor}

\date{}

\newtheorem{Theo}{Theorem}[section]
\newtheorem{Prop}[Theo]{Proposition}

\newtheorem{Lemm}[Theo]{Lemma}

\newtheorem{Conj}[Theo]{Conjeture}

\newcommand{\Inf}{\operatorname{Inf}}
\newcommand{\Var}{\operatorname{Var}}
\newcommand{\BH}{\operatorname{BH}}

\DeclareMathOperator{\Maj}{Maj}
\DeclareMathOperator{\sign}{sign}
\newcommand{\supp}{\operatorname{supp}}
\newcommand{\poly}{\operatorname{poly}}

\def\N{\mathbb{ N}}
\def\R{\mathbb{ R}}

\def\JJ{\mathcal{ J}}
\def\MM{\mathcal{ I}}

\begin{document}

\title[Spectrum of functions on Boolean cubes]{\bf On the Fourier spectrum  of functions\\ on Boolean
cubes}

\author{ A.~Defant, M.~Masty{\l}o and A.~P\'erez}

\maketitle

\noindent
\renewcommand{\thefootnote}{\fnsymbol{footnote}}
 \footnotetext{The second named author was supported by the National Science Centre, Poland,
project no. 2015/17/B/ST1/00064. The research of the third author was partially done during a stay in Oldenburg (Germany) under the support of a PhD fellowship of ``La Caixa Foundation'', and of the projects of MINECO/FEDER (MTM2014-57838-C2-1-P) and Fundaci\'{o}n S\'{e}neca - Regi\'{o}n de Murcia (19368/PI/14).}.

\begin{abstract}
\noindent
Let $f$ be a real-valued, degree-$d$ Boolean function defined on the $n$-dimensional  Boolean cube $\{ \pm 1\}^{n}$, and  $f(x) = \sum_{S \subset \{1,\ldots,d\}} \widehat{f}(S) \prod_{k \in S} x_k$ its Fourier-Walsh expansion. The main result states that there is an absolute constant $C >0$ such that the
$\ell_{2d/(d+1)}$-sum of the Fourier coefficients of $f:\{ \pm 1\}^{n} \rightarrow [-1,1]$ is  bounded by  $\leq C^{\sqrt{d \log d}}$.  It was recently proved that a~similar result holds for complex-valued  polynomials on the $n$-dimensional poly torus $\mathbb{T}^n$, but that in contrast to this a replacement of the $n$-dimensional  torus $\mathbb{T}^n$ by $n$-dimensional cube  $[-1, 1]^n$ leads to a substantially   weaker estimate.
This in the Boolean case  forces us to invent  novel techniques which differ from the ones used in the  complex or real case.
We indicate how  our result is  linked with several questions in quantum information theory.
\end{abstract}




\vspace{10 mm}

\section{Introduction}

The analysis of maps $f:\{ \pm 1\}^{n} \rightarrow \mathbb{R}$ defined on the Boolean cube $\{ \pm 1\}^{n}$ is of interest in theoretical computer sciences, signal processing, social choice, combinatorics, graph theory, as well as others.
For detailed explanation  we refer to the surveys  \cite{RyanSurvey} and \cite{Wolf}, or the more extensive book \cite{BooleanRyan} and the references therein. The study of these functions is carried out through their polynomial expansion, a canonical representation of $f$ in terms of basic functions known as Walsh functions, which are defined for every $S \subset [n]:=\{ 1, \ldots, n\}$ as
\[
\chi_{S}:\{ \pm 1\}^{n} \rightarrow \{ \pm 1\} \, ,
\hspace{3mm} \chi_{S}(x) = x^{S}:= \prod_{k \in S}{x_{k}}, \hspace{3mm} x \in \{\pm 1\}^{n},
\]
where $\chi_{\emptyset}(x):= 1$ for each $x\in \{\pm 1\}^n$.
Using elementary linear algebra it can be checked that every $f$ as above can be written as a~unique linear combination of these $\chi_{S}$. Another more profitable point of view to prove this fact is by means of Fourier analysis on groups: we can look at  $\{\pm 1\}^{n}$ as a~compact abelian group endowed with the coordinatewise product and the  discrete topology (the Haar measure is then given by the normalized counting measure). This way, for each function $f:\{ \pm 1\}^{n} \rightarrow \mathbb{R}$ we have an integral or expectation given by
\[
\mathbb{E}\big[ f \big] := \frac{1}{2^{n}} \sum_{x \in \{ \pm 1\}^{n}}{f(x)}
\]
The dual group of $\{ \pm 1\}^{n}$ actually consists of the set of all Walsh functions $\chi_{S}$ for $S \subset [n]$, which let us associate to each such $f:\{\pm 1\}^{n} \rightarrow \R$ its Fourier-Walsh expansion
\begin{equation}\label{equa:FourierExpansionShort}
f(x) = \sum_{S \subset [n]}{\widehat{f}(S) \, x^S}\, , \hspace{3mm} x \in \{\pm 1\}^{n}\,,
\end{equation}
where the coefficients are explicitly given by $\widehat{f}(S) = \mathbb{E}\big[ f \cdot \chi_{S} \big]$.
Thereby, a~nonzero function $f$ of degree $d$ satisfies that $\widehat{f}(S) = 0$ whenever $|S| > d$, i.e., its Fourier spectrum concentrates on levels less than or equal to $d$. We say that $f$ is $d$-homogeneous if moreover $\widehat{f}(S) = 0$ provided  $|S| \neq d$, i.e., the Fourier spectrum concentrates at level $d$.

\subsection{Main result}
Blei in \cite{Blei01}  proved that for each $d$ there is a~constant $C(d)>0$ such that for each degree-$d$ Boolean function
 $f\colon \{ \pm 1\}^{n}\rightarrow\mathbb{R}$,
\begin{equation}\label{equa:BHBoolean}
\Big(\sum_{\substack{S \subset [n]\\ |S|\leq d}}{|\widehat{f}(S)|^{\frac{2d}{d+1}}}\Big)^{\frac{d+1}{2d}} \leq C(d) \, \| f\|_{\{ \pm 1\}^{n}},
\end{equation}
and here the exponent  $\tfrac{2d}{d+1}$ is best possible  (to denote the supremum norm of a scalar-valued function on $K$ we will often write $\|f\|_K$
instead of $\|f\|_\infty$). The best possible constant $C(d)$ is denoted by
\[
\BH_{\{\pm 1\}}^{\leq d}\,.
\]
The inequality \eqref{equa:BHBoolean} is a Boolean analog of an important inequality for complex polynomials due to Bohnenblust and Hille \cite{BoHi31}. This inequality today forms a  fundamental tool of the theory of Dirichlet series $\sum_n a_n n^{-s}$, and it states  that
for every $d$ there is a (best) constant
$$\BH_{\mathbb{T}}^{\leq d}$$
such for each complex valued degree-$d$ polynomial $P(z) =\sum_{\substack{\alpha \in \mathbb{N}_0^n\\|\alpha|\leq d}}a_\alpha z^\alpha$ in $n$  complex variables  $z_1, \ldots, z_n$
\begin{equation} \label{original}
\Big(\sum_{\substack{\alpha \in \mathbb{N}_0^n\\|\alpha|\leq d}}|a_\alpha|^{\frac{2d}{d+1}}\Big)^{\frac{d+1}{2d}} \leq \BH_{\mathbb{T}}^{\leq d} \, \| f\|_{\mathbb{T}^n};
\end{equation}
again the exponent $\tfrac{2d}{d+1}$ can not be improved. Originally, this result was  proved for
$d$-homogeneous polynomials only. If we define
\[\BH_{\mathbb{T}}^{= d}\,,\]
in the  by now obvious way, a  straight forward trick   shows that
\begin{equation} \label{<d=d}
 \BH_{\mathbb{T}}^{= d} = \BH_{\mathbb{T}}^{\leq d} \,;
\end{equation}
indeed,  for a degree-$d$ polynomial $P(z) =\sum_{|\alpha|\leq d}a_\alpha z^\alpha$ on $\mathbb{T}^n$ the polynomial
$Q(z,w) = \sum_{|\alpha|\leq d}a_\alpha z^\alpha w^{d-|\alpha|}$ is $d$-homogeneous on $\mathbb{T}^{n+1}$ with the same set of coefficients
and the same sup-norm.
\\

In recent years the Bohnenblust-Hille constants $\BH_{\mathbb{T}}^{\leq d}$ were undertaken an intensive study. The original proof of \eqref{original} and its later improvements  give upper estimates
which are essentially of the order $\sqrt{d}^d$, but in \cite{Annals} it was proved that
 $\BH_{\mathbb{T}}^{\leq d}$  grows at most like $\sqrt{2}^d$,
and this  was later substantially improved in \cite{BohrRadiusBPS}, where the authors show that there is an absolute constant $C>0$ such that
\begin{align} \label{BaPeSe1}
\BH_{\mathbb{T}}^{\leq d} \leq C^{\sqrt{d \log d}}.
\end{align}
In this case we say that the growth of $\BH_{\mathbb{T}}^{\leq d}$ is subexponential. In particular this implies that
\begin{align} \label{BaPeSeA1A}
\limsup_{d \rightarrow \infty} \sqrt[d]{\BH_{\mathbb{T}}^{\leq d}}  = 1\,.
\end{align}
 A~natural question appears related to the asymptotic decay of $\BH_{\{\pm 1\}}^{\leq d}$  as a~function in $d$ in the case when  for each $n\in \mathbb{N}$
the group $\mathbb{T}^n$ is replaced by the group  $\{\pm 1\}^n$.
Our main result is the following analog of \eqref{BaPeSe1}.

\begin{Theo} \label{main}
There is an absolute constant $C> 0$ such that for each $d \in \N$
$$\BH_{\{\pm 1\}}^{\leq d} \leq C^{\sqrt{d\log{d}}}.$$
In particular, $\limsup_{d \rightarrow \infty} \sqrt[d]{\BH_{\{\pm 1\}}^{\leq d}}  = 1$.
\end{Theo}

In section \ref{AA} we indicate that this result is closely related with several questions in quantum information theory.\\

Comparing  Bohnenblust-Hille constants for complex polynomials on the  polytorus $\mathbb{T}^n$ with Bohnenblust-Hille
constants for real polynomials on the \linebreak $n$-dimensional cube $[-1,1]^n$ indicates  that both behave substantially
different. Therefore  the preceding result for the Boolean cube $\{\pm 1\}^n$ was not necessarily expected, and in fact its
proof has to overcome new technical difficulties. We intend to explain all this  in the rest of this introduction.\\

\subsection{Real Bohnenblust-Hille constants}
A comparison of our proof of Theorem \ref{main} with that of \eqref{BaPeSe1}  shows that in fact in the setting of Boolean functions  new substantial difficulties appear. Looking at \eqref{equa:FourierExpansionShort} we see that each Boolean function
$f: \{\pm 1\}^n \rightarrow \mathbb{R}$ is the restriction a unique polynomial $P_{f}: \mathbb{R}^{n} \rightarrow \mathbb{R}$ being affine in each variable, a special type of real polynomial sometimes referred as tetrahedral:
\begin{equation} \label{Bild}
\begin{tikzcd}[column sep=small]
\{ \pm 1\}^{n} \arrow[hookrightarrow]{rr}{} \arrow[swap]{dr}{f}& & \,\,\mathbb{R}^{n} \arrow{dl}{P_{f}}\\
& \mathbb{R} &
\end{tikzcd}
\end{equation}
This identification is actually an isometry for the supremum norm since the fact that $P_{f}$ is separately affine gives
\begin{equation} \label{affine}
\| f\|_{\infty} := \sup_{x \in \{ \pm 1\}^{n}}{|f(x)|} = \sup_{x \in \{ \pm 1\}^{n}}{|P_{f}(x)|} = \sup_{x \in [-1,1]^{n}}{|P_{f}(x)|} =: \| P_{f}\|_{\infty}.
\end{equation}
We can then look at these objects from different points of view, allowing us to combine tools and techniques from Fourier analysis, polynomials and multilinear forms in order to study the Fourier spectrum of these functions. Based on results and ideas which have been developed in recent years in the setting of complex polynomials, this is indeed our aim in the present paper.\\

What about the Bohnenblust-Hille inequality for  real polynomials on the $n$-dimensional cube $[-1,1]^n$?
We denote  by $$\BH_{[-1,1]}^{\leq d}$$ the best constant $C\ge 1$ such that for all real(!) degree-$d$ polynomials
$P(x) =\sum_{\substack{\alpha \in \mathbb{N}_0^n\\|\alpha| \leq d}}a_\alpha x^\alpha$ we have
\begin{align} \label{BaPeSe11}
\Big(\sum_{\substack{\alpha \in \mathbb{N}_0^n\\|\alpha|\leq d}}|a_\alpha|^{\frac{2d}{d+1}}\Big)^{\frac{d+1}{2d}} \leq \BH_{[-1,1]}^{\leq d} \, \| f\|_{[-1,1]^n},
\end{align}
and analogously we define for $d$-homogeneous polynomials the constant  $\BH_{[-1,1]}^{= d}$ (with its by now obvious meaning).

\begin{Prop} \label{BvB}
We have
\begin{align} \label{BaPeSeAA}
\limsup_{d \rightarrow \infty} \sqrt[d]{\BH_{[-1,1]}^{= d}}  = 2\,,
\end{align}
and
\begin{align} \label{BaPeSeAAA}
\limsup_{d \rightarrow \infty} \sqrt[d]{\BH_{[-1,1]}^{\leq d}}  = 1+ \sqrt{2}\,.
\end{align}
\end{Prop}
The first equality is due to \cite{RealBHineq} whereas the second  formula seems new.
Interestingly, while the preceding proposition shows that for  general real polynomials the situation is dramatically different from  the complex case \eqref{BaPeSeA1A}, the result from Theorem \ref{main} indicates that tetrahedral polynomials seem to be halfway.\\

In order to give a proof of Proposition \ref{BvB},\eqref{BaPeSeAAA} we need some preliminary results.
 We collect a few crucial inequalities on polynomials which 'regulate the traffic' between the sup norm of
a~real polynomial taken on $[-1,1]^n$ compared with the sup norm of its complexification taken on $\mathbb{T}^n$, as well as the relation of the sup norm of real or complex polynomials in relation to their homogeneous parts.

\begin{Lemm} \label{felix1}
Let $P(z) = \sum_{|\alpha|\leq d} a_\alpha z^\alpha$ be a degree-$d$ polynomial on $\mathbb{C}^n$  with complex coefficients $ a_\alpha$,
and for $m \leq d$ denote by $P_m (z)  = \sum_{|\alpha| = m} a_\alpha z^\alpha$ its $m$-homogeneous part. Then
\begin{itemize}
\item[{\rm(1)}]
$\|P_m\|_{\mathbb{T}^n} \leq \|P\|_{\mathbb{T}^n}$\,
\item[{\rm(2)}]
$\|P\|_{\mathbb{T}^n} \leq (1 + \sqrt{2})^{d} \|P\|_{[-1,1]^n}$\,
\item[{\rm(3)}]
$\|P_d\|_{\mathbb{T}^n} \leq 2^{d-1} \|P\|_{[-1,1]^n}$ whenever  $P$ has real coefficients. Moreover, $C=2$ is best possible.
\item[{\rm(4)}]
$\|P_m\|_{[-1,1]^n} \leq (1 + \sqrt{2})^{d} \|P\|_{[-1,1]^n}$ whenever  $P$ has real coefficients. Moreover $C=1 + \sqrt{2}$
is best possible.
\end{itemize}
\end{Lemm}
Statement (1) is a standard Cauchy estimate. The second result is due to Klimek \cite{Klimek}, and the third one  due to Visser \cite{Visser}.
\\

 Let us turn to the proof of statement (4):
Denote by
$$T_{d}(x) = \sum_{m=0}^{d}{a_m(T_d)} x^{k},\quad x \in \mathbb{R}$$
Chebyshev's polynomial of degree $d$, and recall  that  Markov's numbers $M_{m,d}, 0 \leq m \leq d$ are given by
\begin{align} \label{explicit}
M_{m,d}:=\begin{cases} |a_m(T_d)| & \mbox{ if } m \equiv d  \,  (\mbox{mod } 2)\\
|a_m(T_{d-1})| & \mbox{ if } m \equiv d-1 \,  (\mbox{mod } 2)\,, \end{cases}
\end{align}
where
\[
|a_m(T_d)| = 2^{m-1} \, \frac{d'}{m!} \, \frac{\big( \frac{d'+m-2}{2}\big)!}{\big( \frac{d'-m}{2} \big)!}
\]
with $d' = d$ if  $m \equiv d  \,  (\mbox{mod } 2)$, and  $d' = d-1$ if  $m \equiv d-1 \,  (\mbox{mod } 2)$ (see, e.g., \cite[p.~56]{Natanson64} and also \eqref{Chebby}).
A throughout this article crucial point is the fact that by Markov's theorem (see \cite[p. 248]{Borwein12}) any real polynomial $p(t)=\sum_{m=0}^{d}{a_{m} t^{m}}$ of degree $d$ satisfies
\begin{equation}\label{equa:MarkovTheorem}
|a_{m}| \leq M_{m,d} \sup_{t \in [-1,1]}{|p(t)|}, \quad 0 \leq m \leq d,
\end{equation}
and implementing $T_d$ shows that this estimate  is actually optimal. Given a~polynomial $P$ on $[-1,1]^n$ with real coefficient (as in Proposition \ref{felix1},(4)), we consider for any $x \in \mathbb{R}^{n}$  the polynomial
\[
p_{x}(t) =P(t x_{1}, \ldots, t x_{n}) = \sum_{m=0}^{d}{t^{m} \, P_{m}(x)},\,\,\, t \in\mathbb{R}.
\]
Then, using \eqref{equa:MarkovTheorem} and taking the supremum on $x \in [-1,1]^{n}$,  we easily get that
\begin{align*}
\|P_m\|_{[-1,1]^n} \leq M_{m,d}\,\|P\|_{[-1,1]^n}.
\end{align*}
But the explicit formulas for $M_{m,d}$ from \eqref{explicit} combined with Stirling's formula show that
\begin{equation} \label{poz}
M =\limsup_{d \rightarrow \infty } \Big(\sup_{ m\leq d} \sqrt[d]{M_{m,d}}\Big)\,=\,\sup_{\substack{m,d \in \mathbb{N}\\ m\leq d}} \sqrt[d]{M_{m,d}}\,=\,1 + \sqrt{2}\,.
\end{equation}
This proves  the inequality in (4)  with constant  $(1 + \sqrt{2})^{d}$. Finally, if we assume that this inequality holds with some constant $C^d$, then inserting $T_d$ gives that  $M_{m,d} \leq C^d$ for all $m\leq d$, hence \eqref{poz}  gives $1 + \sqrt{2} \leq C$.\\

Let us note that, as observed in \cite{RealBHineq}, the upper estimate in \eqref{BaPeSeAA} in fact is an immediate consequence of  \eqref{BaPeSe1} combined with Lemma \ref{felix1}(3).
It remains to give the
\begin{proof}[Proof of Proposition \ref{BvB}, \eqref{BaPeSeAAA}] Take a real degree-$d$ polynomial $P$ on $[-1,1]^n$, and block it into its sum
$P(x) = \sum_{m=0}^d P_m(x)$ of homogeneous parts. Then we deduce from Minkowski's inequality
and Lemma \ref{felix1}(1) that
\begin{align*}
\Big(\sum_{\substack{\alpha \in \mathbb{N}_0^n\\|\alpha|
\leq d}}
|a_\alpha|^{\frac{2d}{d+1}}\Big)^{\frac{d+1}{2d}}
&
\leq  \sum_{m=0}^d \Big(\sum_{\substack{\alpha \in \mathbb{N}_0^n\\|\alpha|= m}}|a_\alpha|^{\frac{2m}{m+1}}\Big)^{\frac{m+1}{2m}}
\\&
\leq  \sum_{m=0}^d \BH_{\mathbb{T}}^{=m}\, \|P_m\|_{\mathbb{T}^n}
\leq  \sum_{m=0}^d \BH_{\mathbb{T}}^{=m}\, \|P\|_{\mathbb{T}^n}
\,,
\end{align*}
and then by Lemma \ref{felix1}(2) and equation \eqref{BaPeSe1} that
\begin{align*}
\Big(\sum_{\substack{\alpha \in \mathbb{N}_0^n\\|\alpha|
\leq d}}
|a_\alpha(P)|^{\frac{2d}{d+1}}\Big)^{\frac{d+1}{2d}}
&
\leq  \sum_{m=0}^d \BH_{\mathbb{T}}^{=m} \,\,(1+ \sqrt{2})^d\|P\|_{[-1,1]^n}
\\&
\leq C^{\sqrt{d \log d }} d(1+ \sqrt{2})^d \, \|P\|_{[-1,1]^n}
\,,
\end{align*}
which proves  the upper estimate. Conversely, applying this inequality to Chebyshev's polynomial $T_d$, we see that by \eqref{poz}
\[
1+ \sqrt{2}=\limsup_{d \rightarrow \infty} \Big(\sup_{m\leq d} \sqrt[d]{M_{m,d}}\Big) \leq \limsup_{d \rightarrow \infty} \sqrt[d]{\BH_{[-1,1]}^{\leq d}}\,,
\]
the conclusion.
\end{proof}

\subsection{Sidon constants}
Let us compare \eqref{equa:BHBoolean} and \eqref{original} within the setting of so-called Sidon sets.
Given $1 \leq p \leq \infty$ and a~compact abelian group $G$, a~finite subset $\Lambda$ of the dual group $\widehat{G}$ is said
to be a~$p$-Sidon set if there is a constant $C>0$ (depending on $p$ and $\Lambda$) such that for every trigonometric polynomial
$f = \sum_{\gamma \in \Lambda} \widehat{f}(\gamma) \gamma $ we have
\begin{equation}\label{equa:pSidonSetDefi}
\Big(\sum_{\gamma \in \Lambda}{|\widehat{f}(\gamma)|^{p}}\Big)^{\frac{1}{p}} \leq C \, \| f\|_{G},
\end{equation}
and the so-called Sidon constant $S_p(\Lambda)$ is the best possible $C$.
Here we are  mainly concerned with the groups $\mathbb{T}^n$, the $n$-dimensional polytorus,  and $\{\pm 1\}^n$,
the $n$-dimensional Boolean cube. Recall that  the dual group of $\mathbb{T}^n$ consists of all monomials $z^ \alpha$, $\alpha \in \mathbb{Z}^n$, whereas
the dual group of $\{\pm 1\}^n$ is formed by all monomials $x^S$, $S \subset [n]$. We concentrate on the following four sets of characters in the dual group of
$\mathbb{T}^n$ and $\{\pm 1\}^n$, respectively:
\[
\Omega_n^{=d} = \big\{z^\alpha \colon \alpha \in \mathbb{N}_0^n, \,|\alpha| =d  \big\} \,\,\, \,\text{and} \,\,\,\,\Omega_n^{\leq d} =  \big\{z^\alpha \colon \alpha \in \mathbb{N}_0^n, \,|\alpha| \leq d  \big\}\,,
\]
as well as
\[
\Lambda_n^{=d} = \big\{x^S \colon  S \subset [n], \,|S| =d  \big\} \,\,\,\, \text{and} \,\,\,\,\Lambda_n^{\leq d} = \big\{x^S \colon  S \subset [n], \,|S| \leq d  \big\}\,.
\]
In this terminology we have
\begin{align*}
\BH_{\{\pm 1 \}}^{\leq d} = \sup_n S_{\frac{2d}{d+1}}(\Lambda_n^{\leq d}) < \infty\,,
\end{align*}
and
\begin{align*}
\BH_{\mathbb{T}}^{\leq d} = \sup_n S_{\frac{2d}{d+1}}(\Omega_n^{\leq d}) < \infty\,,
\end{align*}
both  obvious reformulations of \eqref{equa:BHBoolean} and \eqref{original}.
Moreover, Theorem \ref{main}  states that there exists a~constant $C\ge 1$ such that for all $n,d$
\[
S_{\frac{2d}{d+1}}(\Lambda_n^{\leq d}) \leq C^{\sqrt{d \log d}}\,;
\]
we note again that by \eqref{BaPeSe1}  the analog for $S_{\frac{2d}{d+1}}(\Omega_n^{\leq d})$  holds true.

\section{The $d$-homogeneous case} \label{strategy}
Let us come to the proof of Theorem \ref{main}, the Boolean analog of \eqref{BaPeSe1}.
We first prove the $d$-homogeneous case:
There is an absolute constant $C> 0$ such that for every $d \in \N$
\begin{equation}\label{main-hom}
\BH_{\{\pm 1\}}^{= d} \leq C^{\sqrt{d\log{d}}}\,.
\end{equation}
Although Theorem \ref{main} of course covers this case we prefer to give an independent proof
since it is simpler and illustrates the general strategy. Anyway all proofs of Bohnenblust-Hille type inequalities
-- old ones or very recent ones -- are very much inspired by the  proof of the original estimate \eqref{original} which is itself a
 refinement of the proof of  Littlewood's famous $4/3$ inequality from \cite{L30}. Let us recall its four crucial steps:

\begin{itemize}
\item[{\bf S1}]
Given a $d$-homogeneous  polynomial $f(z) = \sum_{|\alpha|=d} a_\alpha z^\alpha$ in $n$ variables, understand  it as the restriction of a symmetric  $d$-linear form
$L_f$ on $\mathbb{C}^n$ to the diagonal $\Delta = \{(z, \ldots, z)\colon z \in \mathbb{C}^n\}$.
\item[{\bf S2}]
Interpret $L_f$ as a symmetric matrix $A = (c_{i_1, \ldots, i_m})_{i_1, \ldots, i_d =1}^n$,  relate  the $\ell_{2d/d+1}$-norm
of the coefficients of $f$ with the $\ell_{2d/d+1}$-norm of the  entries of $A$, and split this norm into mixed terms of the form
$$
\sum_{i_k} \big( \sum_{\substack{i_1, \ldots, i_{k-1}\\i_k, \ldots, i_{d}}} |c_{i_1, \ldots, i_d}|^{2} \big)^{1/2}\,.
$$
\item[{\bf S3}]
 Use Khinchine's inequality to estimate all of these mixed  terms by the supremum norm of $L_f$ on $(\mathbb{T}^n)^d$.
\item[{\bf S4}]
 Finally, use polarization to estimate $\|L_f\|_\infty$ by $\|f\|_\infty$.
\end{itemize}
For  the improvement of the constants, step-by-step non trivial refinements of the preceding arguments were needed.
Modifying several arguments from \cite{Annals} and \cite{BohrRadiusBPS} we now  sketch the proof of \eqref{main-hom}, and start to collect a few crucial ingredients which will
also be needed later for the  more involved proof of Theorem \ref{main}.

\subsection{Multilinear forms}
For each $n \in \N$ and finite set $A \subset \N$ we define the sets
\begin{align*}
\mathcal{I}(A,n) & := \left\{ \mathbf{i}: A \rightarrow [n] \mbox{ map} \right\}, \\
\mathcal{J}(A,n) & := \left\{ \mathbf{j}: A \rightarrow [n]  \mbox{ non-decreasing} \right\}.
\end{align*}
In case $A = [d]$ for some $d \in \N$, we denote the previous sets by $ \mathcal{I}(d,n)$ and $\mathcal{J}(d,n)$ repectively.  We also consider the equivalence relation on $\mathcal{I}(A,n)$ defined for $\mathbf{i}_{1}, \mathbf{i}_{2} \in \mathcal{I}(A,n)$ as
\[ \mathbf{i}_{1} \sim \mathbf{i}_{2} \, \mbox{ if there is a bijection $\sigma: A \rightarrow A$ such that } \mathbf{i}_{1} = \mathbf{i}_{2} \circ \sigma. \]
Note that for each $\mathbf{i} \in \mathcal{I}(A,n)$  there is a unique $\mathbf{j} \in \mathcal{J}(A,n)$ such that $\mathbf{j} \in [\mathbf{i}]$, where $[\mathbf{i}]$ stands for the equivalent class of $\mathbf{i}$.
Moreover the number of elements of the equivalence class of $\mathbf{i} \in \mathcal{I}(A,n)$ is
\begin{equation} \label{againN}
 |[\mathbf{i}]| = \frac{|A|!}{  |\{ k \colon \mathbf{i}(k) = 1 \}| ! \cdot \ldots \cdot |\{ k \colon \mathbf{i}(k) = n \}| !}.
 \end{equation}
If $A_{1}$ and $A_{2}$ are two disjoint subsets of $\mathbb{N}$, then the direct sum of $\mathbf{i}_{1} \in \mathcal{I}(A_{1},n)$ and $\mathbf{i}_{2} \in \mathcal{I}(A_{2}, n)$ is defined as the map
\[ \mathcal{I}(A_{1} \cup A_{2}, n) \, \ni \, \mathbf{i}_{1} \oplus \mathbf{i}_{2}: A_{1} \cup A_{2} \rightarrow [n]  \]
which coincides with $\mathbf{i}_{k}$ on $A_{k}$  for $k=1,2$. Note that every element of $\mathcal{I}(A_{1} \cup A_{2}, n)$ can be represented this way. This yields in particular, that if $S \subset [d]$ and we denote $\hat{S}:=[d] \setminus S$ (its complement in $[d]$), then
$
\mathcal{I}(d, n) = \mathcal{I}(S,n) \oplus \mathcal{I}(\hat{S},n).
$
However we can only assure that
$ \mathcal{J}(d,n) \subset \mathcal{J}(S,n) \oplus \mathcal{J}(\hat{S},n).$
The next result  (sometimes called Blei's inequality) appears with slightly different notation in \cite[Theorem 2.1]{BohrRadiusBPS}.
\begin{Prop}\label{Theo:Blei'sFormula1}
Let $n \in \N$ and $0 \leq m \leq d$ be integers. For any scalar matrix $(a_{\mathbf{i}})_{ \mathbf{i} \in \mathcal{I}(d, n)}$ we have
\[ \left( \sum_{\mathbf{i} \in \mathcal{I}(d, n)}{|a_{\mathbf{i}}|^{\frac{2d}{d+1}}} \right)^{\frac{d+1}{2d}} \leq \left[ \prod_{\substack{S \subset [d]\\ |S| = k}}{\left(  \sum_{\mathbf{i}_{1} \in \mathcal{I}(S, n)}{ \left(  \sum_{\mathbf{i}_{2} \in \mathcal{I}(\hat{S}, n)}{|a_{\mathbf{i}_{1} \oplus \mathbf{i}_{2}}|^{2}} \right)^{\frac{1}{2} \frac{2k}{k+1}} } \right)^{\frac{k+1}{2k} }} \right]^{\frac{1}{\binom{d}{k}}}\,.  \]
\end{Prop}
\noindent Every  $d$-homogeneous function $f:\{\pm 1\}^{n} \rightarrow \R$ has a unique representation
(see again \eqref{equa:FourierExpansionShort})
\begin{equation}
\label{equa:booleanFunctionIndexed}
f(x) = \sum_{\mathbf{j} \in \JJ(d,n)}{a_{\mathbf{j}} x_{\mathbf{j}}} \,,
\hspace{5mm} x_{\mathbf{j}} := x_{j_{1}} \cdot \ldots \cdot x_{j_{d}}\,,
\end{equation}
where $a_{\mathbf{j}} = 0$ if $\mathbf{j}$ is not injective (i.e., not strictly increasing). We can then look at $f$ as the restriction to $\{\pm 1\}^{n}$ of the $d$-homogeneous polynomial $P_{f}:\mathbb{R}^{n} \rightarrow \R$ given by
\begin{equation}
\label{equa:booleanPolynomialIndexed}
P_{f}(x) = \sum_{\mathbf{j} \in \JJ(d,n)}{a_{\mathbf{j}} x_{\mathbf{j}}}\, ,
\end{equation}
and as in \eqref{affine} we have
\begin{equation}
\label{Kanzi}
 \| P_{f}\| = \sup_{x \in [-1,1]^{n}}{|P_{f}(x)|} = \sup_{x \in \{\pm 1\}^{n}}{|P_{f}(x)|} = \| f\|_{\infty}.
 \end{equation}
The real polynomial $P_{f}$ has associated a unique symmetric $d$-linear form $L_{f}: (\R^{n})^{d} \rightarrow \mathbb{R}$ satisfying $P_{f}(x) = L_{f}(x, \ldots, x)$, which for $y^{(1)}, \ldots,y^{(d)}\in \mathbb{R}^n$
can be explicitly written as
\begin{equation}
\label{equa:symmetricMlinearIndexed}
L_f(y^{1}, \ldots, y^{d}) = \sum_{\mathbf{i} \in \mathcal{I}(d,n)}{c_{\mathbf{i}} \: y^{1}_{i_{1}} \cdot \ldots \cdot y^{d}_{i_{d}}}, \hspace{5mm} c_{\mathbf{i}} =\frac{a_{\mathbf{j}}}{d!} \hspace{4mm} \mbox{ whenever $\mathbf{i} \in [\mathbf{j}]$}.
\end{equation}

\subsection{Hypercontractivity}
For each $1 \leq p <  \infty$ the $L_{p}$-norm of the function $f: \{\pm 1\}^n \rightarrow \mathbb{R}$ is given by
\[
\| f\|_{p} = \mathbb{E}\big[|f|^{p} \big]^{1/p} = \mbox{$\left(\sum_{x \in \{ \pm 1\}^{n}}{\frac{1}{2^{n}}|f(x)|^{p}}\right)^{1/p}$}.
\]
A very important map is the noise operator defined for each $-1 < \rho < 1$ as the map $T_\rho$ which acts on each such  $f$  as
\[
T_{\rho}f(x) := \sum_{S \subset [n]}{\widehat{f}(S) \rho^{|S|} x^{S}}, \quad\, x\in \{\pm 1\}^{n}\,.
\]
An important feature of this operator is presented in the next result of Bonami and Gross (see, e.g.,  \cite[Chapter 9]{BooleanRyan}):
For any $1 < p \leq q \leq \infty$ and $\rho \leq \sqrt{\frac{p-1}{q-1}}$,
\[ \| T_{\rho}f\|_{q} \leq \| f\|_{p}. \]
 The following consequence   will be crucial: For for every  degree-$d$ function $f:\{ \pm 1\}^{n} \rightarrow \mathbb{R}$
\begin{equation} \label{hypo}
\left\| f  \right\|_{2} \leq \big(1/\sqrt{p-1}\big)^{d} \, \left\| f \right\|_{p}\,\,, \,\,\, 1 < p \leq 2,
\end{equation}
and
\begin{equation} \label{hypoX}
\left\| f  \right\|_{2} \leq e^d \, \left\| f \right\|_{1}\,.
\end{equation}

\subsection{Polarization I}
We will make use of the following polarization formula (inspired by its complex analog due to Harris  \cite[Theorem 1]{Harris}).

\begin{Prop} \label{Theo:polarizationReal}
Let $P: \R^{n} \rightarrow \mathbb{R}$ be a $d$-homogeneous polynomial, and  $L:(\mathbb{R}^{n})^{d} \rightarrow \mathbb{R}$ its  (unique) associated $d$-linear symmetric form. Then for every $1 \leq k \leq d$ and $x, y \in [-1,1]^{n}$
$$  |L(\overbrace{x, \ldots, x}^{\mbox{$k$}}, \overbrace{y, \ldots, y}^{\mbox{$d-k$}}) | \leq M_{k,d} \, \frac{d^{d}}{k^{k} (d-k)^{d-k}} \, \frac{k! \, (d-k)!}{d!} \, \| P\|_{[-1,1]^n} \,,  $$
where $M_{k,d}$ is the Markov number given in \eqref{explicit}.
\end{Prop}

\begin{proof}
Let $x, y \in [-1,1]^{n}$. Consider the degree-$d$  (real) polynomial
\[ g(t) = P\left(t \, k \, x + (d-k) \, y \right). \]
Using the multinomial formula, we  write
$$
g(t) = \sum_{j=0}^{d}{\binom{d}{j} \, t^{j} k^{j} (d-k)^{d-j} \, L(\overbrace{x, \ldots, x}^{\mbox{$j$}}, \overbrace{y, \ldots, y}^{\mbox{$d-j$}})},
$$
and hence by Markov's inequality  \eqref{equa:MarkovTheorem} we deduce that
$$
k^{k} (d-k)^{d-k} \binom{d}{k} |L(\overbrace{x, \ldots, x}^{\mbox{$k$}}, \overbrace{y, \ldots, y}^{\mbox{$d-k$}})| \leq M_{k,d} \, \sup_{t \in [-1,1]}{|g(t)|} \leq M_{k,d} \, d^{d} \| P\|,
$$
which leads to the desired conclusion.
\end{proof}

\subsection{Proof of the homogeneous case}
For the proof of  \eqref{main-hom} let $f:\{\pm 1\}^{n} \rightarrow \R$ be an $d$-homogeneous function with $d \geq 2$ (the case $d = 1$ is anyway clear). We rewrite it in the form \eqref{equa:booleanFunctionIndexed},
$$ f(x) = \sum_{\mathbf{j} \in \JJ(d,n)}{a_{\mathbf{j}} x_{\mathbf{j}}}\,,$$
where $a_{\mathbf{j}} = 0$ whenever $\mathbf{j}$ is not strictly increasing,
and extend the definition of the $a_{\mathbf{j}}$ to all of $\MM(d,n)$ by $a_{\mathbf{i}} = a_{\mathbf{j}}$ if $\mathbf{i} \in [\mathbf{j}]$. Moreover, for the coefficients $c_{\mathbf{i}}$ of the unique symmetric form  $L_f$
from \eqref{equa:symmetricMlinearIndexed} we have $c_{\mathbf{i}} =\frac{a_{\mathbf{i}}}{d!}$. Then by  Proposition \ref{Theo:Blei'sFormula1} ($a_{\mathbf{i} \oplus \mathbf{j}} =0$
whenever $\mathbf{i}$ or $\mathbf{j}$ are not strictly increasing)
\begin{align*}
\Big(\sum_{\mathbf{i} \in \JJ(d,n)}{|a_{\mathbf{i}}|^{\frac{2d}{d+1}}}& \Big)^{\frac{d+1}{2d}} & \leq \Big( \prod_{\substack{S \subset [d]\\ |S| = k}}{\Big( \sum_{\mathbf{i} \in \JJ(S,n)}{\Big( \sum_{\mathbf{j} \in \JJ(\hat{S},n)}{|a_{\mathbf{i} \oplus \mathbf{j}}|^{2}} \Big)^{\frac{1}{2} \frac{2k}{k+1}}} \Big)^{\frac{k+1}{2k}}} \Big)^{\frac{1}{\binom{d}{k}}}\,.
\end{align*}
Fix $S \subset [d]$ with $|S| = k$ in order to
 estimate each of these factors. If we denote $\rho_{r} := (r-1)^{-\frac{1}{2}}$, then using  \eqref{againN} and \eqref{hypo}, we have
\begin{align*}
\Big( \sum_{\mathbf{i} \in \JJ(S,n)}\Big( \sum_{\mathbf{j} \in \JJ(\hat{S},n)}&{|a_{\mathbf{i} \oplus \mathbf{j}}|^{2}} \Big)^{\frac{1}{2} \frac{2k}{k+1}} \Big)^{\frac{k+1}{2k}}
  \\
  &
  \leq \rho_{\frac{2k}{k+1}}^{d-k} \left( \sum_{\mathbf{i} \in \JJ(S,n)}
  \mathbb{E}_{y} \left[\Big| \sum_{\mathbf{j} \in \JJ(\widehat{S}, n)}{a_{\mathbf{i} \oplus \mathbf{j}} y_{\mathbf{j}}} \Big|^{\frac{2k}{k+1}}  \right]\right)^{\frac{k+1}{2k}}\\
& = \rho_{\frac{2k}{k+1}}^{d-k}
\,\,\,
   \left(\mathbb{E}_{y} \left[
  \sum_{\mathbf{i} \in \JJ(S,n)} \Big| \sum_{\mathbf{j} \in \JJ(\widehat{S}, N)}{a_{\mathbf{i} \oplus \mathbf{j}} y_{\mathbf{j}}} \Big|^{\frac{2k}{k+1}}
  \right]\right)^{\frac{k+1}{2k}}
\\
& =  \rho_{\frac{2k}{k+1}}^{d-k}\, \frac{d!}{(d - k)!}\,
\left(   \mathbb{E}_{y} \left[\sum_{\mathbf{i} \in \JJ(S,N)}{\Big| \sum_{\mathbf{j} \in \mathcal{I}(\widehat{S}, n)}{c_{\mathbf{i} \oplus \mathbf{j}} y_{\mathbf{j}}} \Big|^{\frac{2k}{k+1}}} \right]\right)^{\frac{k+1}{2k}}
\\
&  \leq \rho_{\frac{2k}{k+1}}^{d-k}\,\frac{d!}{(d - k)!}\, \sup_{y \in \{\pm 1\}^n}{\left( \sum_{\mathbf{i} \in \JJ(S,n)}{\Big| \sum_{\mathbf{j} \in \mathcal{I}(\widehat{S}, n)}{c_{\mathbf{i} \oplus \mathbf{j}} y_{\mathbf{j}}} \Big|^{\frac{2k}{k+1}}} \right)^{\frac{k+1}{2k}}}\,.
\end{align*}
Then by  the definition of $\BH_{\{\pm 1\}}^{=k}$ and Proposition \ref{Theo:polarizationReal} we  for every $y \in \{\pm 1\}^n$ get
\begin{align*}
\Big( \sum_{\mathbf{i} \in \JJ(S,n)}{\Big| \sum_{\mathbf{j} \in \mathcal{I}(\widehat{S}, n)}{c_{\mathbf{i} \oplus \mathbf{j}} y_{\mathbf{j}}} \Big|^{\frac{2k}{k+1}}} \Big)^{\frac{k+1}{2k}} & \leq \BH_{\{\pm 1\}}^{=k}\, \sup_{x \in \{\pm 1\}^n}{\Big| \sum_{i \in \JJ(S,n)}{\sum_{j \in \mathcal{I}(\widehat{S},n)}{c_{\mathbf{i} \oplus \mathbf{j}} x_{\mathbf{i}} y_{\mathbf{j}}}} \Big|}\\
& \hspace{-6mm} = \BH_{\{\pm 1\}}^{=k}\, \frac{1}{k!} \, \sup_{x \in \{\pm 1\}^n}{\Big| \sum_{i \in \mathcal{I}(S,n)}{\sum_{j \in \mathcal{I}(\widehat{S},n)}{c_{\mathbf{i} \oplus \mathbf{j}} x_{\mathbf{i}} y_{\mathbf{j}}}} \Big|}\\
& \hspace{-6mm} \leq \BH_{\{\pm 1\}}^{=k} \, \frac{1}{k!} \, \sup_{x,y \in \{\pm 1\}^n}{|L_f(\overbrace{x, \ldots, x}^{\mbox{$k$}}, \overbrace{y, \ldots, y}^{\mbox{$d-k$}})|}\\
& \hspace{-6mm} \leq \BH_{\{\pm 1\}}^{=k} \, M_{k,d} \, \frac{d^{d}}{k^{k} (d-k)^{d-k}} \, \frac{(d-k)!}{d!} \, \|f\|_{\infty}\,.
\end{align*}
Therefore,
$$
\Big( \sum_{\mathbf{i} \in \JJ(S,n)}{\Big( \sum_{\mathbf{j} \in \JJ(\hat{S},n)}{|c_{\mathbf{i} \oplus \mathbf{j}}|^{2}} \Big)^{\frac{1}{2} \frac{2k}{k+1}}} \Big)^{\frac{k+1}{2k}}
\leq  \BH_{\{\pm 1\}}^{=k} \,\rho_{\frac{2k}{k+1}}^{d-k} \, M_{k,d} \, \frac{d^{d}}{k^{k} \, (d-k)^{d-k}} \| f\|_{\infty} ,
$$
which leads to
$$ \BH_{\{\pm 1\}}^{=d} \leq \BH_{\{\pm 1\}}^{=k} \,\rho_{\frac{2k}{k+1}}^{d-k} \, M_{k,d} \, \frac{d^{d}}{k^{k} \, (d-k)^{d-k}}. $$
Replacing the value of $\rho_{\frac{2k}{k+1}}$ and using \eqref{poz}, we get that
for $1 < k < d$,
\begin{align*}
 \BH_{\{\pm 1\}}^{=d} & \leq  \BH_{\{\pm 1\}}^{=k} \Big( \frac{k+1}{k-1} \Big)^{\frac{d-k}{2}}
 (1 + \sqrt{2})^k
  \frac{d^{d}}{k^{k} \, (d-k)^{d-k}}\\
& =  \BH_{\{\pm 1\}}^{=k} \, \Big( 1 + \frac{2}{k-1} \Big)^{\frac{d-k}{2}}
 (1 + \sqrt{2})^k \,\,
  \frac{d^k}{k^k} \Big( 1 + \frac{k}{d-k} \Big)^{d-k}\\
& \leq  \BH_{\{\pm 1\}}^{=k}  \exp{\left(  \frac{d-k}{k-1}\right)}
(1 + \sqrt{2})^k
\exp{\left(k \log{\frac{d}{k}}\right)} \exp{k}
\,,
\end{align*}
where  we have used the fact that $(1 + 1/x)^{x}$ is increasing and bounded by $e$ for positive $x$. Taking $k\approx\sqrt{d/\log{d}}$ and using that $\BH_{\{\pm 1\}}^{=k}$ is non-decreasing in $k$, we obtain
for some $\alpha >0$
\[ \BH_{\{\pm 1\}}^{=d} \leq \BH_{\{\pm 1\}}^{=[\sqrt{d}]} \exp{\big( \alpha \sqrt{d \log{d}} \big)}. \]
Iterating this process yields
\begin{align*}
 \BH_{\{\pm 1\}}^{=d}& \leq e^{\big( \alpha\sqrt{d \log{d}} \big)}  e^{\big( \alpha\sqrt[4]{d} \frac{1}{\sqrt{2} } \sqrt{\log{d}}  \big)} \BH_{\{\pm 1\}}^{=[\sqrt[4]{d}]} \\
& \leq \exp{\Big( \alpha \sqrt{d} \sqrt{\log{d}}\,\Big( 1 + \frac{1}{\sqrt{2}} + \frac{1}{(\sqrt{2})^{2}} + \ldots \Big) \Big)}\BH_{\{\pm 1\}}^{=2}.
\end{align*}
Finally, note that $\BH_{\{\pm 1\}}^{=2} < \infty$ which is a consequence of  \eqref{affine} and \eqref{BaPeSeAA} . \qed

\section{The degree-$d$ case} \label{get in}
Equation \eqref{main-hom} is a result on  $d$-homogeneous tetrahedral polynomials, but we want to pass now to the case of degree-$d$ functions. Let us recall that in the complex case this step is easily done via the trick that was described after  \eqref{<d=d}, but which is not available in the real setting. Another way  (which also works in the complex case) is to split the Fourier expansion of the degree-$d$ function $f$ into its $m$-homogenous parts $f_{m}$ so that
\[
\Big(\sum_{|S| \leq d}{|\widehat{f}(S)|^{\frac{2d}{d+1}}}\Big)^{\frac{d+1}{2d}} \leq C^{\sqrt{d \log{d}}} \sum_{m=0}^{d}{\| f_{m}\|_{\infty}}\,,
\]
and then try to control the norm of $f_{m}$ by the norm of the full function $f$. However, the best estimation we know for tetrahedral polynomials comes from the general estimation from Lemma \ref{felix1} (4) valid for every real polynomial. But this constant in fact has  exponential growth on the degree (in contrast to the complex case in which one gets constant $1$ by Lemma \ref{felix1} (1)). One may wonder whether this constant can be improved in the tetrahedral case. But considering the majority function $\Maj_{d}:\{ \pm 1\}^{d} \rightarrow \{ \pm 1\}, \, \Maj_{d}(x) = \sign (x_1+\ldots, x_d) $ one can easily check that the norm of its $m$-homogeneous part ($m$ odd) satisfies
\[ \|(\Maj_{d})_{m}\|_{\{ \pm 1\}^{d}} = \binom{\frac{d-1}{2}}{\frac{m-1}{2}} \frac{d}{m} \frac{1}{2^{d-1}} \binom{d-1}{\frac{d-1}{2}} \geq \binom{\frac{d-1}{2}}{\frac{m-1}{2}} \frac{\sqrt{d}}{m}\,, \]
so taking $m=(d-1)/2$ we have  $\|(\Maj_{d})_{m}\|_{\infty} = \Omega(2^{d/2} /\sqrt{d})$.
See the problem posed in section \ref{get out}.\\

All of these problems force us to
add to the classical strategy, described in the preceding section, some apparently new techniques in order to prove that in the tetrahedral case we can still  get constants $\BH_{\{\pm 1\}^n}^{\leq d}$ that are  subexponential as in the complex case \eqref{BaPeSe1}. In contrast to Proposition \ref{BvB} the key is that now  Fourier analysis on the Boolean cube $\{\pm 1\}^n$  is at our disposal -- as in the complex case we have the polytorus $\mathbb{T}^n$.

\subsection{Multiaffine forms} \label{tetra}
Let $Q: \mathbb{R}^{n} \rightarrow \mathbb{R}$ be a polynomial of degree-$d$. Recall that $Q$ admits a unique representation known as its monomial expansion: There is a unique
 family of scalars $a_{j_{1} \ldots j_{m}}$ such that for  each $x \in \mathbb{R}^{n}$
\begin{equation}\label{equa:monomialExpansion1}
Q(x) = \sum_{m=0}^{d}{\sum_{1 \leq j_{1} \leq \ldots \leq j_{m} \leq n}{a_{j_{1} \ldots j_{m}} x_{j_{1}} x_{j_{2}} \ldots x_{j_{m}} }}\,.
\end{equation}
 For each $0 \leq m \leq d$, the $m$-homogeneous part of $Q$ is given then by
\[ Q_{m}(x) = \sum_{1 \leq j_{1} \leq \ldots \leq j_{m} \leq n}{a_{j_{1} \ldots j_{m}} x_{j_{1}} \ldots x_{j_{m}}}, \hspace{3mm} x \in \mathbb{R}^{n}. \]
We say that $Q$ is a tetrahedral  polynomial if it has the form
\begin{equation}\label{equa:monomialExpansionTetrahedral1}
Q(x) = \sum_{m=0}^{d}{\sum_{1 \leq j_{1} < \ldots < j_{m} \leq n}{a_{j_{1} \ldots j_{m}} x_{j_{1}} x_{j_{2}} \ldots x_{j_{m}} }},\hspace{4mm} x \in \mathbb{R}^{n}.
\end{equation}
In other words, tetrahedral polynomials are characterized for being affine (a linear map plus a constant) in each variable $x_{j} \in \mathbb{R}$. This is equivalent to say that $Q(x)$ preserves convex combinations in each variable, i.e., if $\mathbf{x}_{1}, \ldots, \mathbf{x}_{n}$ are independent (real) random variables, then
\begin{equation}\label{equa:RandomVatetrahedral}
\mathbb{E}[Q(\mathbf{x}_{1}, \ldots, \mathbf{x}_{n})] = Q(\mathbb{E}[\mathbf{x}_{1}], \ldots, \mathbb{E}[\mathbf{x}_{n}]).
\end{equation}
As we mentioned in the introduction, the importance of these polynomials for us is that there is a~natural (isometric) correspondence between them and real-valued functions on the Boolean cube. Indeed, if for a given function $f:\{ \pm 1\}^{n} \rightarrow \mathbb{R}$ of degree $d$ we put
\begin{equation}\label{equa:coefficientsTetrahedralBoolean}
a_{j_{1} \ldots j_{m}} := \widehat{f}(S) \hspace{3mm} \mbox{ where } \hspace{3mm} S= \{ j_{1} < \ldots < j_{m}\},\quad\, 0 \leq m \leq d\,,
\end{equation}
then the Fourier-Walsh expansion \eqref{equa:FourierExpansionShort} of $f$ can be rewritten as
\[
f(x) = \sum_{m=0}^{d}{\sum_{1 \leq j_{1} < \ldots < j_{m} \leq n}{a_{j_{1} \ldots j_{m}} x_{j_{1}} x_{j_{2}} \ldots x_{j_{m}} }},\hspace{4mm} x \in \{ \pm 1\}^{n}.
\]
We denote by $Q_{f}$ the unique tetrahedral polynomial which coincides with $f$ on $\{ \pm 1\}^{n}$, with their respective monomial and Fourier coefficients related via \eqref{equa:coefficientsTetrahedralBoolean}.\\

Following \cite[pp. 1065--66]{KwapienDecoupling}, every polynomial $Q: \mathbb{R}^{n} \rightarrow \mathbb{R}$ of degree $d$ has a~unique $d$-form $L_{Q}:(\mathbb{R}^{n})^{d} \rightarrow \mathbb{R}$ satisfying the following properties:
\begin{enumerate}
\item[(I)] $L_{Q}$ is (separately) $d$-affine, i.e., it is an affine map in each variable $x^{(j)} \in \mathbb{R}^{n}$.
\item[(II)] $L_{Q}$ is symmetric, i.e., the value of $L_{Q}(x^{(1)}, \ldots, x^{(d)})$ is invariant by permutations of its coordinates

$x^{(1)}, \ldots, x^{(d)}$.
\item[(III)] $L_{Q}(x, \ldots, x) = Q(x)$ for every $x \in \mathbb{R}^{n}$.
\end{enumerate}
For each $x^{(1)}, \ldots, x^{(d)} \in \mathbb{R}^{n}$ this form can be explicitly written via the polarization formula
\begin{align*}
L_{Q}(x^{(1)}, & \ldots, x^{(d)}) \\
& = \mathbb{E}_{\xi}\left[ \sum_{m=0}^{d}{\frac{\xi_{1} \ldots \xi_{d}}{d!}(\xi_{1}
+ \ldots + \xi_{n})^{d-m} Q_{m}{(\xi_{1} x^{(1)} + \ldots + \xi_{d} x^{(d)})}}\right],
\end{align*}
where we are considering expectation over all $\xi \in \{ \pm 1\}^{d}$. Let us notice a couple of observations:
\begin{enumerate}
\item[(i)] If $Q$ is $d$-homogeneous, then $L_{Q}$ is the usual $d$-linear form associated to it
(see \eqref{equa:symmetricMlinearIndexed}).
\item[(ii)] For $f:\{ \pm 1\}^{n} \rightarrow \mathbb{R}$ we will denote by $L_{f}$ the $d$-form associated to its tetrahedral polynomial $Q_{f}$. This form also appears in computer science under the name of (fully) decoupled version of $f$  and block multilinear form; see e.g \cite{RyanZhao16}.
\end{enumerate}

\subsection{Polarization II}
\label{sec:PolynomialsAffineNorms}
The following apparently  new result will be crucial for us. It is an analog of Harris' polarization formula for complex homogeneous polynomials and its real variant from Proposition \ref{Theo:polarizationReal}.  Our  case
is much more limited, since for a~degree-$d$ polynomial $Q$ its associated multiaffine form $L_{Q}$ in each variable just preserves convex combinations and not linear ones. This forces us to look  for a substantially different approach.

\begin{Prop}\label{Prop:estimationTwoBlockDecoupledMain}
Let $Q:\mathbb{R}^{n} \rightarrow \mathbb{R}$ be a polynomial of degree $d$ and $L_{Q}$ its $d$-affine associated form. Then, for each $0 \leq m \leq d/2$ we have that
\begin{equation}\label{equa:estimationTwoBlockDecoupledMain}
\sup_{x,y \in [-1,1]^{n}}{\big| L_{Q}(\underbrace{x, \ldots, x}_{m},
\underbrace{y, \ldots,y}_{d-m}) \big|} \leq  2 d^{m} \sup_{x \in [-1,1]^{n}}{|Q(x)|}.
\end{equation}
\end{Prop}

Before starting the  proof, we need an auxiliary result on interpolation which, as the classical Markov's theorem
from  \eqref{equa:MarkovTheorem}, shows another extremal property of the coefficients of the Chebychev polynomials. In contrast to Markov's result, we are going to consider a different basis for the space of degree-$d$ polynomials. For each $0 \leq m \leq d$ let
\[ \psi_{d,m}(t) = 2^{-d} (1 + t)^{m} (1- t)^{d-m} =  \left( \frac{1+t}{2} \right)^{m} \left( \frac{1-t}{2} \right)^{d - m}\,.  \]

\begin{Prop}\label{Theo:MarkovTypeTheoremPsi}
Let $Q(t)$ be a degree-$d$ polynomial. Then there exist unique scalars $a_{n} = a_{n}(Q)$ $(0 \leq n \leq d)$ such that
\[ Q(t) = \sum_{n=0}^{d}{a_{n} \, \psi_{d,n} (t)}, \quad\, t \in \mathbb{R}.
\]
Moreover, for each $0 \leq n \leq d$
\begin{equation}\label{equa:main1MarkovTypeTheoremPsi}
|a_{n}(Q)| \leq |a_{n}(T_{d})| \, \sup_{t \in [-1,1]}{|Q(t)|}\,,
\end{equation}
where $T_{d}$ is the Chebyshev polynomial of degree-$d$. In particular, we have the optimal estimation
\begin{equation}\label{equa:main2MarkovTypeTheoremPsi}
|a_{n}(Q)| \leq \sum_{m=0}^{\min{\{ d-n, n\}}}{4^{m} \binom{d}{2m} \binom{d-2m}{n-m}}  \cdot \sup_{t \in [-1,1]}{|Q(t)|}.
\end{equation}
\end{Prop}

\begin{proof}
We start proving that every degree-$d$ polynomial $Q(t)$ is a linear combination of the elements $\psi_{d,m}$ ($0 \leq m \leq d$) which authomatically will yield the uniqueness of the coefficients by basic linear algebra. In the process, we will give an explicit expression of these coefficients which will let us prove the inequality \eqref{equa:main1MarkovTypeTheoremPsi}, and finally we will calculate explicitely the coefficients $a_{n}(T_{d})$ to conclude from \eqref{equa:main1MarkovTypeTheoremPsi} that \eqref{equa:main2MarkovTypeTheoremPsi} holds.

We start by fixing $d+1$ arbitrary points (although we will choose them suitably later)
\begin{equation}\label{equa:aux1}
-1 \leq t_{0} < t_{1} < \ldots < t_{d} \leq 1.
\end{equation}
Recall that the associated Lagrange basis polynomials are given by
\begin{equation}\label{equa:aux2}
\Delta_{m}(t) = \prod_{\substack{j=0, j \neq m}}^{d}{\frac{t - t_{j}}{t_{m} - t_{j}}} \, , \hspace{3mm} 0 \leq m \leq d
\end{equation}
which satisfy that
\begin{equation}\label{equa:aux3}
Q(t) = \sum_{m=0}^{d}{Q(t_{m}) \Delta_{m}(t)}, \hspace{3mm} t \in [-1,1].
\end{equation}
Now, we are going to write the polynomials $\Delta_{m}(t)$ in terms of the $\psi_{d,n}(t)$. Note that
\begin{equation*}
\prod_{\substack{j=0, j \neq m}}^{d}{(t - t_{j})} = \prod_{\substack{j=0, j \neq m}}^{d}{\left( \left( \frac{1+t}{2} \right) \left(1-t_{j}\right) - \left( \frac{1-t}{2} \right) \left(1+t_{j} \right)  \right)}.
\end{equation*}
Expanding the last product and gathering terms we arrive to an expression of the form
\begin{equation}\label{equa:aux4}
\prod_{\substack{j=0, j \neq m}}^{d}{(t - t_{j})} = \sum_{n=0}^{d}{\alpha_{m,n} \cdot (-1)^{d-n} \cdot \left( \frac{1+t}{2}\right)^{n} \left( \frac{1-t}{2} \right)^{d-n}}\,,
\end{equation}
where the coefficients $\alpha_{m,n}$ are explicitely given by
\begin{equation}\label{equa:aux5}
\alpha_{m,n} = \sum_{\substack{S \subset [d]_{m}\\ |S| = n}}{\, \prod_{j \in S}{(1 + t_{j})} \, \prod_{j \in [d]_{m} \setminus S}{(1 - t_{j})} } \hspace{2mm} \mbox{ being } \, [d]_{m}:=\{ 1, \ldots, d\} \setminus \{ m\}.
\end{equation}
Replacing \eqref{equa:aux4} in \eqref{equa:aux2}, and the resulting expression in \eqref{equa:aux3}, we conclude that
\begin{equation}\label{equa:aux55}
Q(t)  = \sum_{n=0}^{d}{ \left( \sum_{m=0}^{d}{Q(t_{m}) \frac{\alpha_{m,n}
\cdot (-1)^{d-n} }{\prod_{j \neq m}{(t_{m} - t_{j})}}} \right) \psi_{d, n}(t)}, \quad\, t\in [-1, 1].
\end{equation}
This means that the coefficients $a_{n}$ exist and are explicitly given by
\begin{equation}\label{equa:aux6}
a_{n} = a_{n}(Q) = \sum_{m=0}^{d}{Q(t_{m}) \frac{\alpha_{m,n} \cdot (-1)^{d-n} }{\prod_{j \neq m}{(t_{m} - t_{j})}}}, \hspace{3mm} 0 \leq n \leq d.
\end{equation}
This proves the first part of the theorem. Our aim now is to show that \eqref{equa:main1MarkovTypeTheoremPsi}  holds.
Let us notice first that  the $t_{j}$'s are ordered according to \eqref{equa:aux1}, hence
\[ \prod_{j \neq m}{(t_{m} - t_{j})} = (-1)^{d - m} \prod_{j \neq m}{|t_{m} - t_{j}|}, \]
which let us rewrite \eqref{equa:aux6} as
\begin{equation}\label{equa:aux7}
a_{n} = a_{n}(Q) = (-1)^{n} \, \sum_{m=0}^{d}{Q(t_{m}) \frac{\alpha_{m,n}}{\prod_{j \neq m}{|t_{m} - t_{j}|}} (-1)^{m}}.
\end{equation}
Another useful observation which follows from \eqref{equa:aux5} is that  $\alpha_{m,n} \geq 0$ for every $m, n$. Without loss generality we can assume that $\sup_{t \in [-1,1]}{|Q(t)|} = 1$ by normalizing the polynomial. We choose now explicit points of interpolation
\[ t_{m} = \cos{\left( \frac{m \pi}{d}  \right)}, \quad m=0, \ldots, d\,, \]
which of course satisfy the conditions of \eqref{equa:aux1}. Recall that the Chebyshev polynomial $T_{d}$ of degree $d$ satisfies that $T_{d}(t_{m}) = (-1)^{m}$ for each $0 \leq m \leq d$. Hence by \eqref{equa:aux7}
\[ |a_{n}(Q)| \leq \sum_{m=0}^{d}{ \frac{\alpha_{m,n}}{\prod_{j \neq m}{|t_{m} - t_{j}|}} } = \sum_{m=0}^{d}{ T_{d}(t_{m}) \frac{\alpha_{m,n}}{\prod_{j \neq m}{|t_{m} - t_{j}|}} (-1)^{m}} = (-1)^{n} a_{n}(T_{d}), \]
which implies \eqref{equa:main1MarkovTypeTheoremPsi}. Finaly, we compute the coefficients $a_{n}(T_{d})$ using the expression of the Chebyshev polynomial
\begin{align*}
T_{d}(t) & = \sum_{m=0}^{\lfloor d/2 \rfloor}{\binom{d}{2m}  (t^{2} - 1)^{m} \, t^{d - 2m} } \\
& = \sum_{m=0}^{\lfloor d/2 \rfloor}{\binom{d}{2m} (t-1)^{m} (t+1)^{m} \left( \frac{t-1}{2} + \frac{t+1}{2} \right)^{d - 2m}  } \\
& = \sum_{m=0}^{\lfloor d/2 \rfloor}{ \binom{d}{2m} (t-1)^{m} (t+ 1)^{m} \, \sum_{k=0}^{d - 2m}{\binom{d-2m}{k} \left( \frac{t-1}{2} \right)^{k} \left( \frac{t + 1}{2}\right)^{d - 2m - k}}  }\\
& =  \sum_{m=0}^{\lfloor d/2 \rfloor}{ \sum_{k=0}^{d-2m}{ \binom{d}{2m} \binom{d-2m}{k} 2^{2m} \left( \frac{t-1}{2}\right)^{k+m} \left( \frac{t+1}{2}\right)^{d-m-k} } }\\
& = \sum_{n=0}^{d}{\left( \sum_{m=0}^{\min{\{ d-n, n\}}}{4^{m} \binom{d}{2m} \binom{d-2m}{n-m}} \right)  \left( \frac{t-1}{2}\right)^{n} \left( \frac{t+1}{2}\right)^{d-n} }\,,
\end{align*}
and therefore
\begin{equation} \label{Chebby}
a_{n}(T_{d}) = (-1)^{n} \, \sum_{m=0}^{\min{\{ d-n, n\}}}{4^{m} \binom{d}{2m} \binom{d-2m}{n-m}}\,.
\end{equation}
This  finishes the proof.
\end{proof}

\begin{proof}[Proof of Proposition \ref{Prop:estimationTwoBlockDecoupledMain}]
Fix $x,y \in \{ \pm 1\}^{n}$. Notice that since $L_{Q}$ is affine in each variable, for every $t \in [-1, 1]$ we have that
\begin{equation}\label{equa:CoroEstimationDecoupledAux1}
\begin{split}
P(t) & :=Q\left( \frac{1+t}{2} x + \frac{1-t}{2} y \right)\\
& \hspace{1mm} = \sum_{m=0}^{d}{\binom{d}{m} \left( \frac{1+t}{2} \right)^{m} \left( \frac{1-t}{2} \right)^{d-m} L_{Q}(\underbrace{x, \ldots, x}_{m}, \underbrace{y, \ldots,y}_{d-m})  }.
\end{split}
\end{equation}
Moreover, $P(t)$ is the evaluation of $Q$ at a convex combination of elements in $[-1,1]^{n}$, thus
\[ \sup_{t \in [-1,1]}{|P(t)|} \leq \| Q\|_{\infty}. \]
Without loss of generality we can assume that $\| Q\|_{\infty} = 1$.
An application of Theorem \ref{Theo:MarkovTypeTheoremPsi} to \eqref{equa:CoroEstimationDecoupledAux1} gives that, if $0 \leq m \leq d/2$, then we have the bound
\[
\big| L_{Q}(\underbrace{x, \ldots, x}_{m}, \underbrace{y, \ldots,y}_{d-m}) \big| \leq \frac{1}{\binom{d}{m}} \,\sum_{k=0}^{m}{4^{k} \binom{d}{2k} \binom{d-2k}{m-k}} .
\]
To estimate the last constant from above we will use the following immediate consequence of Stirling approximation formula
\[
(2k)! = \binom{2k}{k} (k!)^{2} \geq \frac{4^{k}}{2 \sqrt{k}}(k!)^{2}, \quad\, k \geq 1.
\]
Then
\begin{align*}
\frac{1}{\binom{d}{m}} \sum_{k=0}^{m}{4^{k} \binom{d}{2k} \binom{d-2k}{m-k}} & = \sum_{k=0}^{m}{4^{k} \frac{m! \,
(d-m)!}{(2k)! \, (m-k)! \, (d-m-k)!}}\\
& \leq 1 + \sum_{k=1}^{m}{2 \, \sqrt{k} \, \binom{m}{k} \binom{d-m}{k}}\\
& \leq 2 \, \sum_{k=0}^{m}{\binom{m}{k} (d-1)^{k}} = 2 d^{m},
\end{align*}
which completes the proof.
\end{proof}

\subsection{Proof of the degree-$d$-case}
For the proof of Theorem \ref{main} we
follow again the general strategy from section \ref{strategy}
 for which we must introduce some more notation. For each $n \in \N$ and each  finite $A \subset \N$ we define the sets
\begin{align*}
\mathcal{I}_0(A,n) & := \left\{ \mathbf{i}: A \rightarrow [n] \cup \{ 0\} \mbox{ map } \right\},\\
\mathcal{J}_0(A,n) & := \left\{ \mathbf{j}: A \rightarrow [n] \cup \{ 0\} \mbox{ non-decreasing map } \right\}.
\end{align*}
In case $A = [d]$ for some $d \in \N$, we denote the previous sets by $ \mathcal{I}_0(d,n)$ and $\mathcal{J}_0(d,n)$ repectively. Given $\mathbf{i}\in \mathcal{I}_0(A,n)$, we write
$ \supp{\mathbf{i}} = \{ k \colon \mathbf{i}(k) \neq 0 \} $
to denote the support of $\mathbf{i}$. We again consider the following equivalence relation: $\mathbf{i}_{1} \sim \mathbf{i}_{2}$ in  $\mathcal{I}_0(A,n)$ whenever there is a bijection $\sigma: A \rightarrow A$ such that $\mathbf{i}_{1} = \mathbf{i}_{2} \circ \sigma$.
For each $\mathbf{i} \in \mathcal{I}_0(A,n)$  there is a unique $\mathbf{j} \in \mathcal{J}_0(A,n)$ such that $\mathbf{j} \in [\mathbf{i}]$, where $[\mathbf{i}]$ as above stands for the equivalent class of $\mathbf{i}$. This identification let us rewrite the monomial expansion \eqref{equa:monomialExpansion1} of a degree-$d$ polynomial $Q: \mathbb{R}^{n} \rightarrow \mathbb{R}$ as
\begin{equation}\label{equa:monomialExpansion2}
Q(x) = \sum_{\mathbf{j} \in \mathcal{J}_0(d, n)}{ a_{\mathbf{j}} \, \prod_{k \in \supp{\mathbf{j}}}{x_{\mathbf{j}(k)}} }, \hspace{5mm}  x \in \mathbb{R}^{n}.
\end{equation}
Recall that tetrahedral polynomials are affine in each variable, which using the notation of \eqref{equa:monomialExpansion2} means that the nonzero coefficients $a_{\mathbf{j}}$ must satisfy that $\mathbf{j}$ is injective on its support. For that, we will say that an element $\mathbf{i} \in \mathcal{I}_0(A,n)$ is affine if  $\mathbf{i}(k) \neq \mathbf{i}(k')$ whenever $k, k' \in \supp{\mathbf{i}}$ are different; note that then every element of $[\mathbf{i}]$ is also affine. As a consequence tetrahedral polynomials have a~monomial expansion of the form
\begin{equation}\label{equa:monomialExpansionTetrahedral2}
Q(x) = \sum_{\substack{\mathbf{j} \in \mathcal{J}_0(d, n)\\ \text{ affine }}}{ a_{\mathbf{j}} \, \prod_{k \in \supp{\mathbf{j}}}{x_{\mathbf{j}(k)}} }, \hspace{5mm}  x \in \mathbb{R}^{n}.
\end{equation}
On the other hand, the associated $d$-affine form of a~polynomial $Q$ with monomial expansion \eqref{equa:monomialExpansion2} is then given by (check I,II,III from section \ref{tetra})
\begin{equation}
L_{Q}(x^{(1)}, \ldots, x^{(d)}) = \sum_{\mathbf{i} \in \mathcal{I}_0(d, n)}{c_{\mathbf{i}} \, \prod_{k \in \supp{\mathbf{i}}}{x^{(k)}_{\mathbf{i}(k)}}}\,,
\end{equation}
 where  $c_{\mathbf{i}} = \frac{a_{\mathbf{j}}}{[\mathbf{j}]}$ if $\mathbf{i} \in  [\mathbf{j}]$.
Again we will use that the cardinality of $[\mathbf{i}], \mathbf{i} \in \mathcal{I}_0(A,n)$ is given by
\begin{equation} \label{again.}
|[\mathbf{i}]| = \frac{|A|!}{|\{ k \colon \mathbf{i}(k) = 0 \}| ! \cdot |\{ k \colon \mathbf{i}(k) = 1 \}| ! \cdot \ldots \cdot |\{ k \colon \mathbf{i}(k) = n \}| !}.
\end{equation}
Moreover, for two  disjoint subsets $A_{1}$ and $A_{2}$ of $\mathbb{N}$, the direct sum of $\mathbf{i}_{1} \in \mathcal{I}_0(A_{1},n)$ and $\mathbf{i}_{2} \in \mathcal{I}_0(A_{2}, n)$ is given by the map
$$ \mathcal{I}_0(A_{1} \cup A_{2}, n) \, \ni \, \mathbf{i}_{1} \oplus \mathbf{i}_{2}: A_{1} \cup A_{2} \rightarrow [n] \cup \{ 0\}.$$
 In particular,
$\mathcal{I}_0(d, n) = \mathcal{I}_0(S,n) \oplus \mathcal{I}_0(\hat{S},n)$
for  $S \subset [d]$ and  $\hat{S}:=[d] \setminus D$. \\

We are now in disposition of proving the main result of this article.

\begin{proof}[Proof of Theorem $\ref{main}$]
Fix $0 \leq m \leq d$.  Let $f:\{\pm 1\}^{n} \rightarrow \R$ be a function of degree $d$ and $Q_{f}$ the associated tetrahedral polynomial. We know from \eqref{equa:monomialExpansionTetrahedral2}  that it has a monomial expansion
\[
Q_{f}(x) = \sum_{\mathbf{j} \in \mathcal{J}(d,n)}{a_{\mathbf{j}} \prod_{k \in \supp{\mathbf{j}}}{x_{\mathbf{j}(k)}} }\,,
\]
where $a_{\mathbf{j}} = 0$ whenever  $\mathbf{j}$ is not affine; we  extend this definition by putting $a_{\mathbf{i}} = a_{\mathbf{j}}$ if $\mathbf{i} \in [\mathbf{j}]$ for $\mathbf{i} \in \mathcal{I}_0(d,n)$, $\mathbf{j} \in \mathcal{J}_0(d,n)$\,.
 Let  $L_{f}:(\mathbb{R}^{n})^{d} \rightarrow \mathbb{R}$ be the  $d$-affine form of $Q_f$ described above. Then for every $S \subset [d]$  and $x^{(1)}, \ldots, x^{(d)} \in \mathbb{R}^{n}$
\begin{align*}
L_{f}(x^{(1)}, \ldots, x^{(d)}) =  \sum_{\mathbf{i}_{1} \in \mathcal{I}_0(S,n)}{ \sum_{\mathbf{i}_{2} \in \mathcal{I}_0(\hat{S},n)}{ c_{\mathbf{i}_{1} \oplus \mathbf{i}_{2}} \, \prod_{k \in \supp{\mathbf{i}_{1}}}{x_{\mathbf{i}_{1}(k)}^{(k)}} \prod_{k \in \supp{\mathbf{i}_{2}}}{x_{\mathbf{i}_{2}(k)}^{(k)}}  }}\,,
\end{align*}
 where
\[ c_{\mathbf{i}_{1} \oplus \mathbf{i}_{2}} = \frac{a_{\mathbf{i}_{1} \oplus \mathbf{i}_{2}}}{| [\mathbf{i}_{1} \oplus \mathbf{i}_{2}] | }  \]
(check again that the right-hand side satisfies the three conditions I, II, III, and recall that $L_f$ is unique).
For $|S| = m$ and $x,y \in [-1,1]^{n}$, evaluating the map $L_{f}$ at the elements $x^{(k)} = x$ if $k \in S$ and $x^{(k)} = y$ if $k \notin S$, we have by the symmetry of $L_{f}$ that
\begin{align*}
L_{f}(\underbrace{x, \ldots, x}_{m}, \underbrace{y, \ldots, y}_{d-m}) & = \sum_{\mathbf{i}_{1} \in \mathcal{I}_0(S,n)} \sum_{\mathbf{i}_{2} \in \mathcal{I}_0(\hat{S},n)}{ c_{\mathbf{i}_{1} \oplus \mathbf{i}_{2}} \, \prod_{k \in \supp{\mathbf{i}_{1}}}{x_{\mathbf{i}_{1}(k)}} \prod_{k \in \supp{\mathbf{i}_{2}}}{y_{\mathbf{i}_{2}(k)}}  }\\
& = \sum_{\mathbf{j}_{1} \in \mathcal{J}_0(S,n)}{\sum_{\mathbf{j}_{2} \in \mathcal{J}_0(\hat{S},n)}{ | [\mathbf{j}_{1}] | \cdot  | [\mathbf{j}_{2}] | \cdot c_{\mathbf{j}_{1} \oplus \mathbf{j}_{2}} \, x_{\mathbf{j}_{1}} \, y_{\mathbf{j}_{2}}} }\\
& =  \sum_{\mathbf{j}_{1} \in \mathcal{J}_0(S,n)}{\sum_{\mathbf{j}_{2} \in \mathcal{J}_0(\hat{S},n)}{ A_{\mathbf{j}_{1} \oplus \mathbf{j}_{2}}} \, x_{\mathbf{j}_{1}} \, y_{\mathbf{j}_{2}}}\,,
\end{align*}
where
\begin{equation}
A_{\mathbf{j}_{1} \oplus \mathbf{j}_{2}}= \frac{| [\mathbf{j}_{1}] | \cdot  | [\mathbf{j}_{2}] |}{|[\mathbf{j}_{1} \oplus \mathbf{j}_{2}]|} \, a_{\mathbf{j}_{1} \oplus \mathbf{j}_{2}}.
\end{equation}
Note that the non-zero coefficients $a_{\mathbf{j}_{1} \oplus \mathbf{j}_{2}}$ satisfy that $\mathbf{j}_{1} \oplus \mathbf{j}_{2}$ is affine, which implies in particular that $\mathbf{j}_{1}, \mathbf{j}_{2}$ are affine, and
so by \eqref{again.} we in this case  can bound
\begin{equation}\label{equa:cardinalEquivalenceClassesControlled}
\frac{|[\mathbf{j}_{1} \oplus \mathbf{j}_{2}]|}{|[\mathbf{j}_{1}]| \cdot |[\mathbf{j}_{2}]|} = \frac{d! \cdot |\{ k \colon \mathbf{j}_{1}(k) = 0\}| \, ! \cdot |\{ k \colon \mathbf{j}_{2}(k) = 0\}| \, !}{(|\{ k\colon (\mathbf{j}_{1} \oplus \mathbf{j}_{2})(k) = 0\}|) \, ! \cdot m! \cdot (d-m)!} \leq \binom{d}{m}\,.
\end{equation}
Using Proposition \ref{Theo:Blei'sFormula1}, we have that
\begin{equation}\label{equa:BleiAux}
\Big(\sum_{\mathbf{j} \in \mathcal{J}_0(d,n)}{|a_{\mathbf{j}}|^{\frac{2d}{d+1}}}\Big)^{\frac{d+1}{2d}}  \leq \left[ \prod_{\substack{S \subset [d]\\ |S| = m}}{\Big( \sum_{\mathbf{j}_{1} \in \mathcal{J}_0(S,n)}{\Big( \sum_{\mathbf{j}_{2} \in \mathcal{J}_0(\hat{S},n)}{|a_{\mathbf{j}_{1} \oplus \mathbf{j}_{2}}|^{2}} \Big)^{\frac{1}{2} \frac{2m}{m+1}}} \Big)^{\frac{m+1}{2m}}} \right]^{\frac{1}{\binom{d}{m}}}\,.\\
\end{equation}
For fixed $S \subset [d]$ with $|S|=m$, we are going to estimate from above each factor in the right-hand side of \eqref{equa:BleiAux}. If we denote $\rho_{r} := (r-1)^{-\frac{1}{2}}$ for $r > 1$, then using \eqref{hypo} and  \eqref{equa:cardinalEquivalenceClassesControlled}, we have that for each $\mathbf{j}_{1} \in \mathcal{J}_0(S,n)$ \begin{align*}
\left(\sum_{\mathbf{j}_{2}}{|a_{\mathbf{j}_{1} \oplus \mathbf{j}_{2}}|^{2}}\right)^{\frac{1}{2}} & = \frac{|[\mathbf{j}_{1} \oplus \mathbf{j}_{2}]|}{ | [\mathbf{j}_{1}] | \cdot  | [\mathbf{j}_{2}] |} \left(\sum_{\mathbf{j}_{2}}{ \left| A_{\mathbf{j}_{1} \oplus \mathbf{j}_{2}} \right|^{2}}\right)^{\frac{1}{2}}\\
& \leq \binom{d}{m}  \left(\sum_{\mathbf{j}_{2}}{ \left| A_{\mathbf{j}_{1} \oplus \mathbf{j}_{2}} \right|^{2}}\right)^{\frac{1}{2}}= \binom{d}{m} \, \left( \mathbb{E}_{x} \, \left| \sum_{\mathbf{j}_{2}}{ A_{\mathbf{j}_{1} \oplus \mathbf{j}_{2}} \, x_{\mathbf{j}_{2}}} \right|^{2}\right)^{\frac{1}{2}}\\
& \leq \binom{d}{m} \, \rho^{d-m}_{\frac{2m}{m+1}} \, \left( \mathbb{E}_{x} \, \left| \sum_{\mathbf{j}_{2} }{A_{\mathbf{j}_{1} \oplus \mathbf{j}_{2}} x_{\mathbf{j}_{2}}} \right|^{\frac{2m}{m+1}} \right)^{\frac{m+1}{2m}}\,.
\end{align*}
Summing over all $\mathbf{j}_{1} \in \mathcal{J}_0(S,n)$ we obtain
\begin{align*}
\begin{split}
\left( \sum_{\mathbf{j}_{1}}{\left( \sum_{\mathbf{j}_{2}}{|a_{\mathbf{j}_{1} \oplus \mathbf{j}_2}|^{2}}  \right)^{\frac{1}{2} \frac{2m}{m+1}}} \right)^{\frac{m+1}{2m}}
& \leq \binom{d}{m} \, \rho_{\frac{2m}{m+1}}^{d-m} \,  \left( \sum_{\mathbf{j}_{1}}{ \mathbb{E}_{x} \left| \sum_{\mathbf{j}_{2} }{ A_{\mathbf{j}_{1} \oplus \mathbf{j}_{2}} x_{\mathbf{j}_{1}} } \right|^{\frac{2m}{m+1}} } \right)^{\frac{m+1}{2m}}\\
& = \binom{d}{m} \, \rho_{\frac{2m}{m+1}}^{d-m} \,  \left( \mathbb{E}_{x} \left[ \sum_{\mathbf{j}_{1}}{  \left| \sum_{\mathbf{j}_{2}}{ A_{\mathbf{j}_{1} \oplus \mathbf{j}_{2}} x_{\mathbf{j}_{1}} } \right|^{\frac{2m}{m+1}} } \right] \right)^{\frac{m+1}{2m}} \\
& \leq \binom{d}{m}  \, \rho_{\frac{2m}{m+1}}^{d-m} \,  \sup_{x \in \{ \pm 1\}^{n}} \left( \sum_{\mathbf{j}_{1}}{  \left| \sum_{\mathbf{j}_{2}}{ A_{\mathbf{j}_{1} \oplus \mathbf{j}_{2}} x_{\mathbf{j}_{1}}} \right|^{\frac{2m}{m+1}} } \right)^{\frac{m+1}{2m}}
\end{split}
\end{align*}
and so the last supremum we bound by Proposition \ref{Prop:estimationTwoBlockDecoupledMain}
\begin{align*}
\sup_{x \in \{ \pm 1\}^{n}}\left( \sum_{\mathbf{j}_{1}}{  \left| \sum_{\mathbf{j}_{2}}{ A_{\mathbf{j}_{1} \oplus \mathbf{j}_{2}} x_{\mathbf{j}_{1}}} \right|^{\frac{2m}{m+1}} } \right)^{\frac{m+1}{2m}} & \leq \BH_{\{\pm 1\}}^{\leq m}
\sup_{x, y \in \{ \pm 1\}^{n}}   \left| \sum_{\mathbf{j}_{1} }{  \sum_{\mathbf{j}_{2} }{ A_{\mathbf{j}_{1} \oplus \mathbf{j}_{2}} x_{\mathbf{j}_{1}} y_{\mathbf{j}_{2}}}  } \right|\\
& = \BH_{\{\pm 1\}}^{\leq m} \, \sup_{x,y \in \{ \pm 1\}^{n}}{\big| L_{f}(\underbrace{x, \ldots, x}_{m}, \underbrace{y, \ldots, y}_{d-m}) \big|}\\
& \leq \BH_{\{\pm 1\}}^{\leq m} \, 2 \, d^{m} \, \| f\|_{\infty}\,.
\end{align*}
Hence
\[ \left( \sum_{\mathbf{j}_{1}}{\Big( \sum_{\mathbf{j}_{2} }{|a_{\mathbf{j}_{1} \oplus \mathbf{j}}|^{2}} \Big)^{\frac{1}{2} \frac{2m}{m+1}}} \right)^{\frac{m+1}{2m}}
\leq \BH_{\{\pm 1\}}^{\leq m} \, \binom{d}{m} \, \rho_{\frac{2m}{m+1}}^{d-m} \,   2 \, d^{m} \, \| f\|_{\infty}.  \]
Since the fixed set $S$ was arbitrary, we can apply this bound to each factor in the right-hand side of \eqref{equa:BleiAux} to obtain that
\begin{align*}
\Big(\sum_{\mathbf{j} \in \mathcal{J}(d,n)}{|a_{\mathbf{j}}|^{\frac{2d}{d+1}}}\Big)^{\frac{d+1}{2d}}  \leq \BH_{\{\pm 1\}}^{\leq m} \, \binom{d}{m} \, \rho_{\frac{2m}{m+1}}^{d-m} \,   2 \, d^{m} \, \| f\|_{\infty}\,,
\end{align*}
and consequently
\begin{align*}
\BH_{\{\pm 1\}}^{\leq d} & \leq \BH_{\{\pm 1\}}^{\leq m} \,\rho_{\frac{2m}{m+1}}^{d-m} \, \binom{d}{m} \, 2 d^{m} =\BH_{\{\pm 1\}}^{\leq m} \,\Big( \frac{m+1}{m-1} \Big)^{\frac{d-m}{2}} \, \binom{d}{m} \, 2 d^{m}\\
& \leq \BH_{\{\pm 1\}}^{\leq m} \left( 1+ \frac{2}{m-1} \right)^{\frac{d-m}{2}}  d^{2m} \leq \BH_{m} \exp{\left( \frac{d-m}{m-1} + 2 m \log{d} \right)}\\
& \leq \BH_{\{\pm 1\}}^{\leq m} \exp{\left( 2 \left( \frac{d}{m+1} + m \log{d} \right)\right)}\,.
\end{align*}
Taking $m=[\sqrt{d/\log{d}}]$ in the previous inequality leads to
\begin{align*}
\BH_{\{\pm 1\}}^{\leq d} & \leq
\BH_{\{\pm 1\}}^{\leq [\sqrt{d/ \log{d}}]}
 \exp{\left( 4 \sqrt{d \log{d}} \right)}\,,
\end{align*}
and iterating (as in the proof of \eqref{main-hom}) concludes the argument.
\end{proof}

\section{Remarks and open problems}

\subsection{Aaronson-Ambainis conjecture} \label{AA}
Theorem \ref{main} proves that $ \BH^{\leq d}_{\{\pm 1\}}$ grows at most subexponetially in the degree $d$, but it might even be true that it at most grows
polynomially, i.e., $ \BH^{\leq d}_{\{\pm 1\}} \leq \poly{d}$.
In the following we intend to explain how the Bohnenblust-Hille cycle of ideas is linked with quantum information theory.\\

Given $n,k \in \mathbb{N}$, a function $f: (\{\pm 1\}^{n})^{k}  \rightarrow \mathbb{R}$
 is said to be a multilinear form if it is linear in each
input, i.e., for all $x_i \in \mathbb{R}^n$ and $y \in \mathbb{R}^n$ (and similarly for the other positions)
\[
f(x_1+y,x_2, \ldots, x_k) = f(x+y,x_2, \ldots, x_k)  +f(y,x_2, \ldots, x_k)\,.
\]
 Clearly, any such $k$-homogeneous Boolean function can be written as
 \[ f(x^{1}, \ldots, x^{k}) = \sum_{1 \leq i_{1}, \ldots, i_{k} \leq n}{a_{i_{1} \ldots i_{k}} x_{i_{1}}^{1} \ldots x_{i_{k}}^{k}}, \hspace{4mm} x^{1}, \ldots x^{k} \in \{ \pm 1 \}^{n}\,, \]
and
Montanaro   proved in \cite{Monta} that
\begin{equation}\label{equa:multilinearBH}
\Big( \sum_{1 \leq i_{1}, \ldots, i_{k} \leq n}{|a_{i_{1} \ldots i_{k}}|^{\frac{2k}{k+1}}} \Big)^{\frac{k+1}{2k}} \leq k^{C} \| f\|_{\infty}\,, \end{equation}
where $C>0$ is an absolute constant. He moreover applies this estimation to XOR games in the setting of quantum information theory, as well as to solve a very special case of the
Aaronson-Ambainis conjecture from \cite{ArAm}:

\begin{Conj}\label{Conj:A-A}
Every function $f:\{\pm 1\}^{n} \rightarrow [-1, 1]$ of degree $d$ satisfies that
\begin{equation}\label{equa:A-A}
\poly(\Var(f)/d) \leq \, \max_{1 \leq j \leq n}{ \Inf_{j}(f)}\,,
\end{equation}
where $\Var{(f)} := \sum_{S \neq \emptyset}{\widehat{f}(S)^{2}}$ and $\Inf_{j}(f) := \sum_{j \in S}{\widehat{f}(S)^{2}}$ for every $j \in [n]$.
\end{Conj}
If this conjecture were true, it would imply
(informally) that all quantum query algorithms could be efficiently simulated by classical query
algorithms on most inputs.
It is not clear whether or not there is a direct link between this conjecture and the possible fact that the best constant in the polynomial Bohnenblust-Hille inequality $\BH^{\leq d}_{\{\pm 1\}}$ for degree-$d$ polynomials grows polynomially on $d$. We highline two facts connecting both questions:\\

 \noindent (i) Following Montanaro's approach from \cite[Corollary 20]{Monta}, if $\BH^{\leq d}_{\{\pm 1\}}$ grows polynomially on $d$, then conjecture \ref{Conj:A-A} holds for all functions $f:\{ \pm 1\}^{n} \rightarrow [-1,1]$ of degree $d$ satisfying that $\widehat{f}(S) \in \{ \pm \alpha\}$ for every $S$ ($\alpha \geq 0$). Indeed, we have
\[  \Var{(f)} = \alpha^{2} \sum_{m=1}^{d}{\binom{n}{m}}\,,\]
and for each $j \in [n]$
\[
\Inf_{j}{(f)} = \alpha^{2} \sum_{m=1}^{d}{\binom{n-1}{m-1}} =  \frac{\alpha^2}{n} \sum_{m=1}^{d}{m \binom{n}{m}}\,.
\]
Moreover, using that  $\binom{n}{d} \geq (\frac{n}{d})^{d}$ we get
\[ \alpha \left[ \frac{n}{d} \sum_{m=1}^{d}{\binom{n}{m}} \right]^{\frac{1}{2}} \leq \alpha \left[\sum_{m=1}^{d}{\binom{n}{m}}\right]^{\frac{d+1}{2d}} \leq \Big(\sum_{S \neq \emptyset}{|\widehat{f}(S)|^{\frac{2d}{d+1}}}\Big)^{\frac{d+1}{2d}} \leq \BH^{\leq d}_{\{\pm 1\}}. \]
All together this implies
\[
\frac{\Var(f)^{2}}{\max_{1 \leq j \leq n}{\Inf_{j}(f)}} \leq \alpha^{2} n \sum_{m=1}^{d}{\binom{n}{m}} \leq d \big(\BH^{\leq d}_{\{\pm 1\}}\big)^{2}.
\]
(ii) By
\cite{DiFrKiDo} conjecture \ref{Conj:A-A} is known to hold whenever  we   in \eqref{equa:A-A} replace $d$  by $2^{d}$. A short proof of this fact based on hypercontractivity can be found in \cite{RyanZhao16}, and there is a close resemblance between the main idea of this proof and the  proofs of  Bohnenblust-Hille type inequalities. Indeed, given $f:\{ \pm 1\}^{n} \rightarrow [-1,1]$ of degree $d$ we have that
\begin{equation}\label{equa:A-AexponentialAux}
\Var{(f)} \leq \sum_{j=1}^{n}{\Inf_{j}(f)} \leq \sqrt{\max_{1 \leq j \leq n}{\Inf_{j}(f)}} \, \sum_{j=1}^{n}{\sqrt{\Inf_{j}(f)}}.
\end{equation}
Following the notation of the proof of Theorem \ref{main}, it is clear that, if  $S = \{ d\}$ and $S' = [d-1]$, then
\[ \sum_{j=1}^{n}{\sqrt{\Inf_{j}(f)}} \leq \sum_{\mathbf{j}_{1} \in \mathcal{J}(S,n)}{\Big(\sum_{\mathbf{j}_{2} \in \mathcal{J}(S',n)}{|a_{\mathbf{j}_{1} \oplus \mathbf{j}_{2}}|^{2}}\Big)^{\frac{1}{2}}}\,, \]
and we already know from that proof that
\[ \sum_{\mathbf{j}_{1} \in \mathcal{J}(S,n)}{\Big(\sum_{\mathbf{j}_{2} \in \mathcal{J}(S',n)}{|a_{\mathbf{j}_{1} \oplus \mathbf{j}_{2}}|^{2}}\Big)^{\frac{1}{2}}} \leq \binom{d}{1} \rho^{d-1}_{1}2 d \leq d^{2} e^{d} . \]
Hence, applying this estimation to \eqref{equa:A-AexponentialAux} we get as desired
\[\frac{\Var(f)^{2}}{\max_{1 \leq j \leq n}{\Inf_{j}(f)}} \leq d^{4} e^{4d}\,.\]
\

\subsection{Lorentz norms}
From \cite{DeMa3} we know that for  each positive integer $d$ the Lorentz space $\ell_{\frac{2d}{d+1},1}$ is the
smallest symmetric Banach sequence space $X$ for which there is a Bohnenblust-Hille inequality (replacing
on the left-hand side of the inequality the $\ell_{\frac{2d}{d+1}}$-norm by the $X$-norm).

\begin{Theo}\label{Theo:LorentzBHIneq}
There is $L \ge 1$ such that every function  $f: \{\pm 1\}^n  \rightarrow \mathbb{R}$
of degree-$d$ satisfies
\begin{equation}\label{equa:LorentzBHIneq}
\sum_{k=1}^{D(n,d)} f^{*}(k) \,\frac{1}{k^{\frac{d-1}{2d}}} \leq L^d \, \|f\|_\infty\,,
\end{equation}
where $D(n,d)= \sum_{m=0}^{d}\binom{n}{m}$ and $\big( f^{*}(k) \big)_{k=1}^{D(n,d)}$ is the decreasing rearrangement  of
$\big(|\widehat{f}(S)| \big)_{\substack{S \subset [n]\\|S| \leq d}}$\,.
\end{Theo}

\noindent Note that the norm on the left-hand side of \eqref{equa:LorentzBHIneq} is the norm in the Lorentz space
$$\ell_{\frac{2d}{d+1},1} \Big( \{ S \colon S \subset [n],|S| \leq d\}\Big)
= \Big( \prod_{\substack{S \subset [n]\\|S| \leq d}} \mathbb{R}, \|\cdot\|_{\frac{2d}{d+1},1}\Big)\,,$$
and hence this inequality generalizes the  BH-inequality for functions on the Boolean cube -- with the constraint that we do not know whether the constant in \eqref{equa:LorentzBHIneq} still behaves subexponentially.
If $C(d)$  denotes the best possible constant in \eqref{equa:LorentzBHIneq} (replacing $L^d$), then the following proof shows that
$\limsup_{d\rightarrow \infty} C(d)^{1/d}\leq  \sqrt{2}(1 + \sqrt{2})$.

\begin{proof}[Proof of Theorem \ref{Theo:LorentzBHIneq}]
Let $f:\{ \pm 1\}^{n} \rightarrow \mathbb{R}$ be a polynomial  of degree $d$. If $Q_{f}$ is the corresponding  tetrahedral polynomial, then considered
as a~complex polynomial (of degree $d$) it follows from \cite[Theorem 16]{DeMa3} that for every $\varepsilon > 0$ there is a~constant $C(\varepsilon)$ depending just on $\varepsilon$ such that
\begin{equation}\label{equa:LorentzBHIneqAux1}
\Big\|  \big(  \widehat{f}(S) \big)_{\substack{S \subset [n]\\|S| \leq d}} \Big\|_{\ell_{\frac{2d}{d+1},1}} \leq C(\varepsilon) (\sqrt{2} + \varepsilon)^d  \sup_{z \in \mathbb{D}^n}{|Q_{f}(z)|}.
\end{equation}
Note that the  result in \cite[Theorem 16]{DeMa3} is claimed for $d$-homogeneous functions. But the result easily extends to the degree-$d$-case; indeed, every degree-$d$ polynomial $Q_{f}$ on $[-1,1]^n$ can be viewed  as a $d$-homogeneous polynomial
on $\mathbb{T}^{n+1}$  with the the same complex norm and the same coefficients:  consider  $\omega \,Q_{f}(z_{1}\overline{\omega}, \ldots, z_{n} \overline{\omega} )$. On the other hand, recall
Klimek's result from Lemma \ref{felix1}(4)
\begin{equation*}\label{equa:LorentzBHIneqAux2}
\sup_{z \in \mathbb{D}^{n}}{|Q_{f}(z)|} \leq (1 + \sqrt{2})^{d} \, \sup_{x \in [-1,1]}{|Q_{f}(x)|}.
\end{equation*}
Combining this with \eqref{equa:LorentzBHIneqAux1} we get \eqref{equa:LorentzBHIneq}.
\end{proof}

\subsection{Norm estimates for homogeneous parts} \label{get out}
In view of
Lemma \ref{felix1},(4) and the difficulties described at the beginning of
section \ref{get in} it would be interesting to know whether $C= 1 + \sqrt{2}$ is the best constant satisfying
$$ \| f_{m}\|_{\infty} \leq C^{d} \| f\|_{\infty} $$
for every $f:\{ \pm 1\}^{n} \rightarrow \mathbb{R}$ of degree $d$ and $0 \leq m \leq d$. See also the commend made at the end of section \ref{get in}.

\vspace{2.5 mm}

\noindent
Institut f\"ur Mathematik \\
Carl von Ossietzky Universit\"at \\
Postfach 2503 \\
D-26111 Oldenburg, Germany

\vspace{0.5 mm}

\noindent E-mail: andreas.defant@uni-oldenburg.de

\vspace{2.5 mm}

\noindent Faculty of Mathematics and Computer Science\\
Adam Mickiewicz University in Pozna\'n\\
Umultowska 87\\
61-614 Pozna{\'n}, Poland

\vspace{0.5 mm}

\noindent E-mail: mastylo$@$amu.edu.pl

\vspace{2.5 mm}

\noindent
Departamento de Matem\'{a}ticas \\
Universidad de Murcia, Espinardo \\
30100 Murcia (Spain)

\vspace{0.5 mm}

\noindent
E-mail: antonio.perez7@um.es
\end{document}